\documentclass[11pt,twoside]{article}
\usepackage{amsfonts}
\usepackage{amsmath}
\usepackage{amssymb}
\usepackage{multicol}
\usepackage{graphics}
\usepackage{stmaryrd}
\usepackage{cite}
\newtheorem{theorem}{Theorem}[section]
\newtheorem{lemma}[theorem]{Lemma}

\newtheorem{definition}[theorem]{Definition}
\newtheorem{remark}[theorem]{Remark}

\numberwithin{equation}{section}
\newenvironment{proof}[1][Proof]{\noindent\textbf{#1.} }{\hfill $\Box$}
\allowdisplaybreaks

 \makeatletter
\begin{document}
\author{Guo Lin$^1$\thanks{E-mail: ling@lzu.edu.cn.} \thanks{Supported by NSF of China (11101194).} and Haiyan Wang$^2$\thanks{E-mail: Haiyan.Wang@asu.edu}  \\
{$^1$School of Mathematics and Statistics, Lanzhou University,}\\
{Lanzhou, Gansu 730000, People's Republic of China}\\
{$^2$School of Mathematical and Natural Sciences,}\\
{Arizona State University, Phoenix, AZ 85069-7100}}
\title{\textbf{Traveling Wave Solutions of a Reaction-Diffusion Equation with State-Dependent Delay}}\maketitle
\begin{abstract}
This paper is concerned with the traveling wave solutions of a reaction-diffusion equation with state-dependent delay. When the birth function is monotone, the existence and nonexistence of monotone traveling wave solutions are established. When the birth function is not monotone, the minimal wave speed of nontrivial traveling wave solutions is obtained. The results are proved by the construction of upper and lower solutions and application of the fixed point theorem.


\end{abstract}

\section{Introduction}

Much research efforts have been devoted to understanding the dynamics of differential equations with state-dependent delay.  Differential equations with state-dependent delay are useful to study population dynamics in which the amount of food available per biomass for a fixed food supply is a function of the total consumer biomass \cite{afw}. Andrewartha and Birch \cite[pp. 370]{and} studied a differential equation with state-dependent in which the duration of larval development of flies is a nonlinear increasing function of larval density.  There is a wealth of literature on the research and we refer to Arino et al. \cite{ari}, Cooke and Huang \cite{coo}, Hartung et al. \cite{har}, Hu et al. \cite{hu}, Magal and Arino \cite{magal}, Mallet-Paret and Nussbaum \cite{mal}, Walther \cite{wa} and the references cited therein. Despite that the dynamics of functional differential equations has been widely studied, the spatial-temporal patterns of differential equation with state-dependent delay are hardly understood.

In this paper, we consider the existence and nonexistence of traveling wave solutions of the following reaction-diffusion equation with state-dependent delay
\begin{equation}\label{1}
u_{t}(x,t)= u_{xx}(x,t)-d u(x,t)+b(u(x,t-\tau(u(x,t)))),
\end{equation}
where $x\in\mathbb{R}, t>0.$ In population dynamics, $d>0$ is the death rate and $b: [0,\infty)\to [0,\infty)$ is the birth function satisfying
\begin{description}
  \item[(B)] $b(0)=0,b(K)=K$ for some positive constant $K>0,$ and $b(u)>du, u\in (0,K)$ while $0< b(u)<du, u>K,$
\end{description}
which will be imposed throughout this paper.
In this model, time delay is not a constant and $\tau: [0,\infty)\to [0,\infty)$ satisfies the following assumptions:
\begin{description}
  \item[(A1)] $\tau(u)$ is $C^1$ for $u\in [0,\infty);$
  \item[(A2)] $0\le \tau'(u)<1, u\in [0,\infty);$
  \item[(A3)] $0\le m =\tau (0)\le \lim_{u\to\infty}\tau(u)=M<\infty.$
\end{description}

To continue our discussion, we first present the following definition.
\begin{definition}\label{de1}
{\rm
A traveling wave solution of \eqref{1} is a special entire solution defined for all $x,t\in\mathbb{R}$ and taking the following form
\[
u(x,t)=\phi (\xi), \xi =x+ct \in\mathbb{R},
\]
where $c>0$ is the wave speed and $\phi \in C^2(\mathbb{R}, \mathbb{R})$ is the wave profile. In particular, if $\phi(\xi)$ is monotone increasing, then it is a traveling wavefront.}
\end{definition}

By the above definition, $\phi$ and $c$ must satisfy the following functional differential equation of second order
\begin{equation}\label{2}
\phi''(\xi)-c \phi'(\xi)-d \phi(\xi) +b(\phi(\xi-c\tau(\phi(\xi))))=0,\xi\in\mathbb{R}.
\end{equation}
In particular, to model precise transition processes in evolutionary systems by traveling wave solutions, we also consider the following asymptotic boundary conditions
\begin{equation}\label{3}
\lim_{\xi\rightarrow -\infty }\phi (\xi)=0,\lim_{\xi\rightarrow +\infty
}\phi (\xi)=K
\end{equation}
or a weaker version
\begin{equation}\label{4}
\lim_{\xi\rightarrow -\infty }\phi (\xi)=0,\liminf_{\xi\rightarrow +\infty
}\phi (\xi)>0.
\end{equation}

When $\tau'(u)=0,$ the traveling wave solutions of \eqref{1} have been widely studied since Schaaf \cite{schaaf}.
For monotone $b(u)$ with $\tau'(u)=0,$ comparison principle is admissible,  the existence and nonexistence of monotone traveling wave solutions can be studied by monotone iteration, fixed point theorem and monotone semiflows, see Liang and Zhao \cite{liangzhao}, Ma \cite{ma01}, Wu and Zou \cite{wuzou2}. If $b(u)$ is locally monotone and $\tau'(u)=0$ holds, then the traveling wave solutions of \eqref{1} can be studied by constructing auxiliary monotone equations, see Fang and Zhao \cite{fz}, Ma \cite{ma1}, Wang \cite{why}.

However, when $\tau(u)$ is not a constant, because $\tau(u)$ is nondecreasing, the comparison principle in \eqref{1} does not hold even if $b(u)$ is monotone increasing. Therefore, the study of traveling wave solutions of \eqref{1} needs some new techniques. In this paper, similar to that in Wu and Zou \cite{wuzou2}, we first introduce an integral operator to study the existence of traveling wave solutions.
Based on some estimations, we confirm the comparison principle on a proper subset of the space of continuous functions if $b(u)$ is monotone. Then the existence of \eqref{2}-\eqref{3} is proved by combining Schauder's fixed point theorem with upper and lower solutions if the wave speed is larger than a threshold $c^*$ defined later. When the wave speed is $c^*,$ we establish the
existence of \eqref{2}-\eqref{3} by passing to a limit. If the wave speed is less than $c^*,$ the nonexistence of \eqref{2}-\eqref{3} is confirmed. Therefore, $c^*$ is the minimal wave speed of \eqref{2}-\eqref{3}.

If $b(u)$ is not monotone,  similar to Ma \cite{ma1}, we introduce two auxiliary monotone equations to confirm the  existence of \eqref{2} with \eqref{4} for $c>c^*$. If $c=c^*,$ the existence of nontrivial solutions of \eqref{2} is also proved by passing to a limit function. When $c<c^*,$ we obtain an auxiliary equation with fixed time delay by \eqref{4}. Applying the theory of asymptotic spreading, we establish the nonexistence of traveling wave solutions.

\section{Monotone Birth Function}

In this section, we investigate the existence of traveling wavefronts of \eqref{1}, namely, existence of monotone solutions of \eqref{2}-\eqref{3}, if
$b: [0,\infty)\to [0,\infty)$ satisfies the following assumptions:
\begin{description}
  \item[(B1)] $b(u)$ is $C^1$ for $u\in [0,K];$
  \item[(B2)] $b'(0)>d$ and $b(u)<b'(0)u, u\in [0,K];$
  \item[(B3)] $0\le b'(u)\le b'(0),u\in [0,K];$
  \item[(B4)] $0\le b'(0)u-b(u)< Lu^2$ for $u\in [0,K]$ and some $L>0.$
\end{description}
Clearly, if $b(u)=pue^{-u},$ then (B1)-(B4) hold when $1<p/d \le e.$ In particular, (B1)-(B4) will be imposed throughout this section without further illustration. Furthermore, we also denote
\[
T=\sup_{u\in [0,K]}\tau '(u).
\]

\subsection{Some Estimations}

Let $\beta >d$ hold and define $H:C_{[0,K]}\rightarrow C$ by
\begin{equation}
H(\phi )(\xi)=\beta \phi (\xi) -d \phi(\xi) + b(\phi(\xi-c\tau(\phi(\xi)))).
\label{3.3}
\end{equation}
Further define an operator $F:C_{[0,K]}\rightarrow C$ by
\begin{equation}
F(\phi )(\xi)=\frac 1{\gamma _2 (c) -\gamma _1 (c) }\left[ \int_{-\infty
}^{\xi}e^{\gamma _1 (c) (\xi-s)} +\int_{\xi}^\infty e^{\gamma
_2 (c)(\xi-s)}\right]H(\phi )(s)ds ,  \label{3.5}
\end{equation}
where
\[
\gamma _1 (c) =\frac{c-\sqrt{c^2+4\beta }}{2},\gamma _2 (c) =\frac{c+\sqrt{%
c^2+4\beta }}{2}.
\]

Clearly, a fixed point of $F$ is a solution of \eqref{2}, and a solution of \eqref{2} is a fixed point of $F.$ Therefore, to study the existence of solutions of \eqref{2}, it suffices to investigate the existence of fixed points of $F.$

\begin{lemma}\label{le2.1}
If $\phi(\xi)\in C_{[0,K]},$ then $0\le H(\phi )(\xi) \le  \beta K.$
\end{lemma}
\begin{lemma}\label{le2.2}
Assume that  $c>0$ is fixed and
\[
4 \beta \ge (1+b^{\prime }(0)KT)^2 c^2 + 4d.
\]
If $\phi(\xi)\in C_{[0,K]},$ then $|cF(\phi) '(\xi)| < \frac{\beta K}{1+ b^{\prime }(0)KT}.$
\end{lemma}
\begin{proof}
By the definition of $F,$ we have
\begin{eqnarray*}
&&c\left\vert \frac{d}{d\xi }F(\phi )(\xi )\right\vert  \\
&=&c\left\vert \frac{1}{\gamma _{2}(c)-\gamma _{1}(c)}\left[ \gamma
_{1}(c)\int_{-\infty }^{\xi }e^{\gamma _{1}(c)(\xi -s)}+\gamma
_{2}(c)\int_{\xi }^{\infty }e^{\gamma _{2}(c)(\xi -s)}\right] H(\phi
)(s)ds\right\vert  \\
&<&c\max \left\{ \left\vert \frac{\gamma _{2}(c)}{\gamma _{2}(c)-\gamma
_{1}(c)}\int_{\xi }^{\infty }e^{\gamma _{2}(c)(\xi -s)}H(\phi
)(s)ds\right\vert ,\right.  \\
&&\left. \left\vert \frac{\gamma _{1}(c)}{\gamma _{2}(c)-\gamma _{1}(c)}%
\int_{-\infty }^{\xi }e^{\gamma _{1}(c)(\xi -s)}H(\phi )(s)ds\right\vert
\right\}  \\
&\leq &\frac{c\beta K}{\gamma _{2}(c)-\gamma _{1}(c)} \\
&=&\frac{\beta cK}{\sqrt{c^{2}+4\beta }} \\
&<& \frac{\beta K}{1+ b^{\prime }(0)KT}.
\end{eqnarray*}
The proof is complete.
\end{proof}
\begin{lemma}\label{le2.3}
Assume that $c>0$ is fixed.
If
\[
\beta \ge d (1+b^{\prime }(0)KT) +\frac{(1+b^{\prime }(0)KT)^2 c^2 }{4},
\]
then
\[
\beta \phi (\xi) -d \phi(\xi) + b(\phi(\xi)-c\tau(\phi(\xi)))
\]
is monotone nondecreasing in $\phi(\xi)\in C_{[0,K]}$ provided that
\[
c|\phi (\xi_1)-\phi(\xi_2)| \le  \frac{\beta K}{1+ b^{\prime }(0)KT}  |\xi_1-\xi_2|, \xi_1,\xi_2\in \mathbb{R}.
\]
\end{lemma}
\begin{proof}
Let $\phi_1(\xi),\phi_2(\xi)$ satisfy $\phi_1(\xi)\ge \phi_2(\xi)$ and
\begin{eqnarray*}
c|\phi _{1}(\xi _{1})-\phi _{1}(\xi _{2})| &\leq &\frac{\beta K}{1+ b^{\prime }(0)KT}|\xi _{1}-\xi
_{2}|, \xi_1,\xi_2\in \mathbb{R}, \\
c|\phi _{2}(\xi _{1})-\phi _{2}(\xi _{2})| &\leq &\frac{\beta K}{1+ b^{\prime }(0)KT}|\xi _{1}-\xi
_{2}|, \xi_1,\xi_2\in \mathbb{R}.
\end{eqnarray*}
By (B3), we have
\[
b(\phi _{1}(\xi -c\tau (\phi _{1}(\xi ))))\ge b(\phi _{2}(\xi -c\tau (\phi
_{1}(\xi ))))
\]
and
\begin{eqnarray*}
&&b(\phi _{1}(\xi -c\tau (\phi _{1}(\xi ))))-b(\phi _{2}(\xi -c\tau (\phi
_{2}(\xi )))) \\
&=&b(\phi _{1}(\xi -c\tau (\phi _{1}(\xi ))))-b(\phi _{2}(\xi -c\tau (\phi
_{1}(\xi )))) \\
&&+b(\phi _{2}(\xi -c\tau (\phi _{1}(\xi ))))-b(\phi _{2}(\xi -c\tau (\phi
_{2}(\xi )))) \\
&\geq &b(\phi _{2}(\xi -c\tau (\phi _{1}(\xi ))))-b(\phi _{2}(\xi -c\tau
(\phi _{2}(\xi )))) \\
&\geq &-b^{\prime }(0)\left\vert \phi _{2}(\xi -c\tau (\phi _{1}(\xi
)))-\phi _{2}(\xi -c\tau (\phi _{2}(\xi )))\right\vert  \\
&\geq &-\frac{b^{\prime }(0)K\beta}{1+ b^{\prime }(0)KT} \left\vert \tau (\phi _{1}(\xi ))-\tau
(\phi _{2}(\xi ))\right\vert  \\
&\geq &-\frac{b^{\prime }(0)KT\beta}{1+ b^{\prime }(0)KT}\left( \phi _{1}(\xi )-\phi _{2}(\xi
)\right).
\end{eqnarray*}
Then the result is clear by the definition of $\beta.$
The proof is complete.
\end{proof}

\begin{lemma}\label{le2.4}
Assume that $c>0$ is fixed.
If
\[
\beta \ge d (1+b^{\prime }(0)KT) +\frac{(1+b^{\prime }(0)KT)^2 c^2 }{4},
\]
then
\[
\beta \phi (\xi) -d \phi(\xi) + b(\phi(\xi)-c\tau(\phi(\xi)))
\]
is monotone nondecreasing in $\xi$ provided that $\phi(\xi)\in C_{[0,K]}$  is monotone nondecreasing in $\xi\in\mathbb{R}$ and
\[
c|\phi (\xi_1)-\phi(\xi_2)| \le  \frac{\beta K}{1+ b^{\prime }(0)KT}|\xi_1-\xi_2|, \xi_1,\xi_2\in \mathbb{R}.
\]
\end{lemma}
\begin{proof}
If $\phi(\xi)\in C_{[0,K]}$  is monotone nondecreasing in $\xi\in\mathbb{R},$ then
\[
\phi (\xi _{1})-\phi (\xi
_{2})\ge 0
\]
for any $\xi_1 \ge \xi_2.$ From the monotonicity of $b,$ we further have
\[
b(\phi (\xi _{1}-c\tau (\phi (\xi _{1}))))\ge b(\phi (\xi _{2}-c\tau (\phi
(\xi _{1}))))
\]
and
\begin{eqnarray*}
&&b(\phi (\xi _{1}-c\tau (\phi (\xi _{1}))))-b(\phi (\xi _{2}-c\tau (\phi
(\xi _{2})))) \\
&=&b(\phi (\xi _{1}-c\tau (\phi (\xi _{1}))))-b(\phi (\xi _{2}-c\tau (\phi
(\xi _{1})))) \\
&&+b(\phi (\xi _{2}-c\tau (\phi (\xi _{1}))))-b(\phi (\xi _{2}-c\tau (\phi
(\xi _{2})))) \\
&\geq &b(\phi (\xi _{2}-c\tau (\phi (\xi _{1}))))-b(\phi (\xi _{2}-c\tau
(\phi (\xi _{2})))) \\
&\geq &-b^{\prime }(0)\left\vert \phi (\xi _{2}-c\tau (\phi (\xi
_{1})))-\phi (\xi _{2}-c\tau (\phi (\xi _{2})))\right\vert  \\
&\geq &-\frac{b^{\prime }(0)K\beta}{1+ b^{\prime }(0)KT}\left\vert \tau (\phi (\xi _{1}))-\tau
(\phi (\xi _{2}))\right\vert  \\
&\geq &-\frac{b^{\prime }(0)KT\beta}{1+ b^{\prime }(0)KT}\left( \phi (\xi _{1})-\phi (\xi
_{2})\right) .
\end{eqnarray*}
Then the result is clear by the definition of $\beta.$ The proof is complete.
\end{proof}

\subsection{Upper and Lower Solutions}

To investigate the existence of \eqref{2}-\eqref{3}, we will use the upper and lower solutions defined as follows.
\begin{definition}\label{de2}
{\rm
A  continuous function ${\overline{\phi }}(\xi) \in
C_{{[0,K]}}(\mathbb{R},\mathbb{R})$ is an
upper solution of \eqref{2}  if there exist constants
$T_{i}, i=1,\cdots ,l$, such that  ${\phi }''(\xi), {\phi }'(\xi)$ exist
and are bounded for all $\xi\in \mathbb{R}\backslash
\{T_{i}:i=1,\cdots ,l\}$ and satisfy
\begin{equation}
{\overline{\phi} }^{\prime \prime }(\xi)-c{\overline{\phi} }^{\prime }(\xi)  -d \overline{\phi}(\xi) + b(\overline{\phi}(\xi-c\tau(\overline{\phi}(\xi))))\leq 0,\text{ } \xi\in
\mathbb{R}\backslash \{T_{i}:i=1,\cdots ,l\}.   \label{3.1}
\end{equation}
A lower solution can be similarly defined by inversing the inequality.
}
\end{definition}

To construct upper and lower solutions, we define some constants. For $\lambda\ge 0,c\ge 0,$ we first define
\[
\Lambda (\lambda,c)=\lambda^2 -c \lambda -d+ b'(0)e^{-\lambda m}.
\]
\begin{lemma}\label{le3.1}
There exists $c^* >0$ such that $\Lambda (\lambda,c) >0$ if $c\in [0,c^*),\lambda\ge 0$ while $\Lambda (\lambda,c) =0$ has two positive distinct real roots $\lambda_1(c) <\lambda_2 (c)$ if $ c>c^*.$ Moreover, $\Lambda (\lambda,c)$ also satisfies
\[
\Lambda (\lambda,c)
\begin{cases}
<0, \lambda \in (\lambda_1(c) ,\lambda_2 (c)),\\
>0, \lambda \in [0, \lambda_1(c))\cup (\lambda_2 (c),\infty).
\end{cases}
\]
\end{lemma}
\begin{lemma}\label{le3.2}
There exist  $L_1, L_2>0$ such that
\[
c \lambda_1(c) < L_1, \lambda_1(c) < L_2, c > c^*.
\]
\end{lemma}
\begin{proof}
Evidently, we have
\[
\Lambda (2(b'(0)-d)/c,c) <0
\]
if $c>0$ is large. Then the result is clear and we complete the proof.
\end{proof}

We now  consider the existence of traveling wave solutions if $c>c^*$ is a constant and define continuous functions as follows
\begin{equation}\label{5}
\overline{\phi }(\xi )=\min \{e^{\lambda _{1}(c)\xi },K\},\underline{\phi }%
(\xi )=\min \{e^{\lambda _{1}(c)\xi }-qe^{\eta \lambda _{1}(c)\xi },0\},
\end{equation}
where
\[
1< \eta< \min\{2, \frac{\lambda_2 (c)}{\lambda_1 (c)}\}, q= \frac{3L_{2}b^{\prime }(0)+L}{- \Lambda (\eta\lambda_1(c),c)}+1+K>1+K.
\]
\begin{lemma}\label{le3.3}
For any $c>c^*,$ there exists $\beta_1 >0$ such that
\begin{eqnarray*}
c|\overline{\phi }(\xi _{1})-\overline{\phi }(\xi _{2})| &\leq &\frac{\beta K}{1+ b^{\prime }(0)KT}|\xi _{1}-\xi _{2}|, \\
c|\underline{\phi }(\xi _{1})-\underline{\phi }(\xi _{2})| &\leq &\frac{\beta K}{1+ b^{\prime }(0)KT}|\xi _{1}-\xi _{2}|
\end{eqnarray*}
for any $\beta \ge \beta_1, c>c^*, \xi_1,\xi_2\in\mathbb{R}.$
\end{lemma}
\begin{proof}
Let $L_1$ be defined by Lemma \ref{le3.2}. If $\overline{\phi }(\xi )=e^{\lambda _{1}(c)}<K,$ then%
\[
c\overline{\phi }^{\prime }(\xi )=c\lambda _{1}(c)e^{\lambda _{1}(c)\xi
}<L_{1}K.
\]%
When $\overline{\phi }(\xi )=K<e^{\lambda _{1}(c)},$ then $c\overline{\phi }%
^{\prime }(\xi )=0$ and
\[
c|\overline{\phi }(\xi _{1})-\overline{\phi }(\xi _{2})| \leq  3L_{1}K |\xi _{1}-\xi _{2}|, \xi_1,\xi_2\in\mathbb{R}.
\]

If $\underline{\phi }(\xi )=e^{\lambda _{1}(c)}-qe^{\eta \lambda _{1}(c)}>0,$
then
\begin{eqnarray*}
\left\vert c\underline{\phi }^{\prime }(\xi )\right\vert  &=& c \left\vert
\lambda _{1}(c)e^{\lambda _{1}(c)\xi }-q\eta \lambda _{1}(c)e^{\eta \lambda
_{1}(c)\xi }\right\vert  \\
&\leq & c \left\vert \lambda _{1}(c)e^{\lambda _{1}(c)\xi }\right\vert
+c \left\vert q\eta \lambda _{1}(c)e^{\eta \lambda _{1}(c)\xi }\right\vert  \\
&<&3L_{1}K,
\end{eqnarray*}%
which implies that
\[
c|\underline{\phi }(\xi _{1})-\underline{\phi }(\xi _{2})| \leq   3L_{1}K |\xi _{1}-\xi _{2}|, \xi_1,\xi_2\in\mathbb{R}.
\]
Let
\[
3L_{1}K=\frac{\beta_1 K}{1+ b^{\prime }(0)KT}.
\]
Then the proof is complete.
\end{proof}
\begin{remark}\label{re1}
{\rm
By what we have done,  we can \emph{fix} $\beta >d$ such that Lemmas \ref{le2.1}-\ref{le3.3}
hold for each \emph{fixed}  $c>c^*.$}
\end{remark}

\begin{lemma}\label{le3.4}
If $\overline{\phi }(\xi )$ and $ \underline{\phi }(\xi )$ are defined by \eqref{5}, then $\overline{\phi }(\xi )$ is an upper solution of \eqref{2} while $\underline{\phi }(\xi )$ is a lower solution of \eqref{2}.
\end{lemma}
\begin{proof}
If $\overline{\phi }(\xi )=e^{\lambda _{1}(c)\xi }<K,$ then
\begin{eqnarray*}
&&{\overline{\phi }}^{\prime \prime }(\xi )-c{\overline{\phi }}^{\prime
}(\xi )-d\overline{\phi }(\xi )+b(\overline{\phi }(\xi -c\tau (\overline{%
\phi }(\xi )))) \\
&\leq &{\overline{\phi }}^{\prime \prime }(\xi )-c{\overline{\phi }}^{\prime
}(\xi )-d\overline{\phi }(\xi )+b^{\prime }(0)\overline{\phi }(\xi -c\tau (%
\overline{\phi }(\xi ))) \\
&\leq &{\overline{\phi }}^{\prime \prime }(\xi )-c{\overline{\phi }}^{\prime
}(\xi )-d\overline{\phi }(\xi )+b^{\prime }(0)\overline{\phi }(\xi -cm) \\
&=&e^{\lambda _{1}(c)\xi }\Delta (\lambda _{1}(c),c) \\
&=&0.
\end{eqnarray*}%
When $\overline{\phi }(\xi )=K<e^{\lambda _{1}(c)\xi },$ then $\overline{%
\phi }(\xi -c\tau (\overline{\phi }(\xi )))\leq K$ and%
\begin{eqnarray*}
&&{\overline{\phi }}^{\prime \prime }(\xi )-c{\overline{\phi }}^{\prime
}(\xi )-d\overline{\phi }(\xi )+b(\overline{\phi }(\xi -c\tau (\overline{%
\phi }(\xi )))) \\
&\leq &-dK+b(K)=0.
\end{eqnarray*}

If $\underline{\phi }(\xi )=0>e^{\lambda _{1}(c)\xi }-qe^{\eta \lambda
_{1}(c)\xi },$ then%
\begin{eqnarray*}
&&\underline{{\phi }}^{\prime \prime }(\xi )-c\underline{{\phi }}^{\prime
}(\xi )-d\underline{\phi }(\xi )+b(\underline{\phi }(\xi -c\tau (\underline{%
\phi }(\xi )))) \\
&=&\underline{{\phi }}^{\prime \prime }(\xi )-c\underline{{\phi }}^{\prime
}(\xi )-d\underline{\phi }(\xi )+b(\underline{\phi }(\xi -cm)) \\
&\ge &0.
\end{eqnarray*}
If $\underline{\phi }(\xi )=e^{\lambda _{1}(c)\xi }-qe^{\eta \lambda
_{1}(c)\xi }>0$, then%
\begin{eqnarray*}
\left\vert \underline{\phi }^{\prime }(\xi )\right\vert  &=&\left\vert
\lambda _{1}(c)e^{\lambda _{1}(c)\xi }-q\eta \lambda _{1}(c)e^{\eta \lambda
_{1}(c)\xi }\right\vert  \\
&<&\lambda _{1}(c)e^{\lambda _{1}(c)\xi }+\eta \lambda _{1}(c)qe^{\eta
\lambda _{1}(c)\xi } \\
&<&3\lambda _{1}(c)e^{\lambda _{1}(c)\xi }
\end{eqnarray*}%
and%
\begin{eqnarray*}
&&\underline{{\phi }}^{\prime \prime }(\xi )-c\underline{{\phi }}^{\prime
}(\xi )-d\underline{\phi }(\xi )+b(\underline{\phi }(\xi -c\tau (\underline{%
\phi }(\xi )))) \\
&\geq &\underline{{\phi }}^{\prime \prime }(\xi )-c\underline{{\phi }}%
^{\prime }(\xi )-d\underline{\phi }(\xi )+b^{\prime }(0)\underline{\phi }%
(\xi -c\tau (\underline{\phi }(\xi ))) \\
&&-L\underline{\phi }^{2}(\xi -c\tau (\underline{\phi }(\xi ))) \\
&=&\underline{{\phi }}^{\prime \prime }(\xi )-c\underline{{\phi }}^{\prime
}(\xi )-d\underline{\phi }(\xi )+b^{\prime }(0)\underline{\phi }(\xi -cm) \\
&&+b^{\prime }(0)\left[ \underline{\phi }(\xi -c\tau (\underline{\phi }(\xi
)))-\underline{\phi }(\xi -cm)\right] -L\underline{\phi }^{2}(\xi -c\tau (%
\underline{\phi }(\xi ))) \\
&\geq &-\Delta (\eta \lambda _{1}(c),c)qe^{\eta \lambda _{1}(c)\xi } \\
&&-3\lambda _{1}(c)b^{\prime }(0)e^{2\lambda _{1}(c)\xi }-Le^{2\lambda
_{1}(c)\xi } \\
&\geq &-\Delta (\eta \lambda _{1}(c),c)qe^{\eta \lambda _{1}(c)\xi
}-(3L_{2}b^{\prime }(0)+L)e^{2\lambda _{1}(c)\xi } \\
&>&0
\end{eqnarray*}%
by the definition of $q.$ The proof is complete.
\end{proof}

\subsection{Existence of Monotone Traveling Wave Solutions: $c>c^*$}

In this part, we assume that $c>c^*$ is fixed and prove the existence of fixed points of $F$ by Schauder's fixed point theorem. Moreover, $\beta >0$ is a fixed constant satisfying Remark \ref{re1}.  Let $\mu \in (0, - \gamma_1(c))$ be a constant and
\[
B_\mu \left( \mathbb{R},\mathbb{R}\right) =\left\{ \phi \in C\left( %
\mathbb{R},\mathbb{R}\right) :\sup_{\xi\in \mathbb{R}}\left| \phi
(t)\right| e^{-\mu \left| \xi\right| }<\infty \right\}
\]
and
\[
\left| \phi \right| _\mu =\sup_{\xi\in \mathbb{R}}\left| \phi
(\xi)\right|
e^{-\mu \left| \xi\right| }\text{ for }\phi \in B_\mu \left( \mathbb{R},%
\mathbb{R}\right) .
\]
Then it is easy to show that $\left( B_\mu \left( \mathbb{R},\mathbb{R}%
\right) ,\left| \cdot \right| _\mu \right) $ is a Banach space.

Define
\[
\Gamma \left( \underline{\phi },\overline{\phi }\right) =\left\{
\phi (\xi)\in B_\mu \left( \mathbb{R},\mathbb{R}\right):
\begin{array}{l}
(i)\text{ }\underline{\phi }(\xi)\leq \phi (\xi)\leq \overline{\phi }(\xi); \\
(ii)\text{ }\phi (\xi)\text{ is nondecreasing;} \\
(iii)\text{ } c |\phi (\xi_1)-\phi(\xi_2)| \le \frac{\beta K}{1+ b^{\prime }(0)KT} |\xi_1-\xi_2|, \xi_1,\xi_2\in \mathbb{R}
\end{array}
\right\}.
\]

By what we have done, we see that $\Gamma$ is nonempty and convex. Moreover, we can verify that $\Gamma$ is closed and bounded with respect to the norm $|\cdot|_{\mu}.$
\begin{lemma}\label{le3.6}
$F$ admits the following nice properties.
\begin{description}
  \item[(1)] If $\phi(\xi)\in \Gamma,$ then $F(\phi)(\xi)$ is monotone increasing.
  \item[(2)] If $\phi_1(\xi),\phi_2(\xi)\in \Gamma$ with $\phi_1(\xi)\le \phi_2(\xi),$ then $F(\phi_1)(\xi)\le F(\phi_2)(\xi).$
  \item[(3)] If $\phi(\xi)\in \Gamma,$ then $c |(F(\phi)(\xi))'| <\frac{\beta K}{1+ b^{\prime }(0)KT}.$
\end{description}
\end{lemma}

Lemma \ref{le3.6} is evident by Lemmas \ref{le2.1}-\ref{le2.3}, and we omit the proof details here.

\begin{lemma}\label{le3.7}
$F: \Gamma \to \Gamma.$
\end{lemma}
\begin{proof}
From Lemma \ref{le3.6}, it suffices to verify that
\[
\underline{\phi } (\xi) \le F(\underline{\phi})(\xi)\le F(\overline{\phi})(\xi)\le \overline{\phi }(\xi), \xi\in \mathbb{R}.
\]

By the definition of upper solution, we see that
\begin{eqnarray*}
&&F\left( \overline{\phi }\right) (\xi ) \\
&=&\frac{1}{\gamma _{2}(c)-\gamma _{1}(c)}\left[ \int_{-\infty }^{\xi
}e^{\gamma _{1}(c)(\xi -s)}+\int_{\xi }^{\infty }e^{\gamma _{2}(c)(\xi -s)}%
\right] H\left( \overline{\phi }\right) (s)ds \\
&\leq &\frac{1}{\gamma _{2}(c)-\gamma _{1}(c)}\left[ \int_{-\infty }^{\xi
}e^{\gamma _{1}(c)(\xi -s)}+\int_{\xi }^{\infty }e^{\gamma _{2}(c)(\xi -s)}%
\right] \left( \beta \overline{\phi }(s)+c\overline{\phi }^{\prime }(s)-%
\overline{\phi }^{\prime \prime }(s)\right) ds \\
&=&\underline{\phi }(\xi )-\frac{\min \left\{ \lambda _{1}(c)Ke^{\gamma
_{2}(c)\left( \xi -\frac{\ln K}{\lambda _{1}(c)}\right) },\lambda
_{1}(c)Ke^{\gamma _{1}(c)\left( \xi -\frac{\ln K}{\lambda _{1}(c)}\right)
}\right\} }{\gamma _{2}(c)-\gamma _{1}(c)} \\
&\leq &\underline{\phi }(\xi ),\text{ }\xi \neq \frac{\ln K}{\lambda _{1}(c)}%
.
\end{eqnarray*}
Then the continuity implies that
\[
\underline{\phi } (\xi) \le F(\underline{\phi})(\xi), \xi\in \mathbb{R}.
\]

In a similar way, we can prove that
\[
F(\overline{\phi})(\xi)\le \overline{\phi }(\xi), \xi\in\mathbb{R}.
\]
The proof is complete.
\end{proof}

\begin{lemma}\label{le3.8}
$F: \Gamma \to \Gamma$ is complete continuous in the sense of $|\cdot|_{\mu}.$
\end{lemma}
\begin{proof}
If $\phi_1,\phi_2\in \Gamma,$ then
\begin{eqnarray*}
&&\left\vert b(\phi _{1}(\xi -c\tau (\phi _{1}(\xi ))))-b(\phi _{2}(\xi
-c\tau (\phi _{2}(\xi ))))\newline
\right\vert \\
&\leq &\left\vert b(\phi _{1}(\xi -c\tau (\phi _{1}(\xi ))))-b(\phi _{2}(\xi
-c\tau (\phi _{1}(\xi ))))\newline
\right\vert \\
&&+\left\vert b(\phi _{2}(\xi -c\tau (\phi _{1}(\xi ))))-b(\phi _{2}(\xi
-c\tau (\phi _{2}(\xi ))))\right\vert \\
&\leq &b^{\prime }(0)\left\vert \phi _{2}(\xi -c\tau (\phi _{1}(\xi )))-\phi
_{1}(\xi -c\tau (\phi _{1}(\xi )))\right\vert \\
&& + b^{\prime }(0)\beta KcT\left\vert \phi _{2}(\xi )-\phi _{1}(\xi
)\right\vert
\end{eqnarray*}
and so
\begin{eqnarray*}
&&\left\vert F(\phi _{1})(\xi )-F(\phi _{2})(\xi )\right\vert  \\
&\leq &\frac{1}{\gamma _{2}(c)-\gamma _{1}(c)}\left[ \int_{-\infty }^{\xi
}e^{\gamma _{1}(c)(\xi -s)}ds+\int_{\xi }^{\infty }e^{\gamma _{2}(c)(\xi -s)}%
\right] \left\vert H(\phi _{1})(s)-H(\phi _{2})(s)\right\vert ds \\
&\leq &\frac{1}{\gamma _{2}(c)-\gamma _{1}(c)}\left[ \int_{-\infty }^{\xi
}e^{\gamma _{1}(c)(\xi -s)}ds+\int_{\xi }^{\infty }e^{\gamma _{2}(c)(\xi -s)}%
\right]  \\
&&\left[ \left( b^{\prime }(0)\beta KcT+\beta -d\right) \left\vert \phi
_{2}(s)-\phi _{1}(s)\right\vert +\right.  \\
&&\left. b^{\prime }(0)\left\vert \phi _{2}(s-c\tau (\phi _{1}(s)))-\phi
_{1}(s-c\tau (\phi _{1}(s)))\right\vert \right] ds.
\end{eqnarray*}%
Applying these estimations, we have
\begin{eqnarray*}
&&\left\vert F(\phi _{1})(\xi )-F(\phi _{2})(\xi )\right\vert e^{-\mu
\left\vert \xi \right\vert } \\
&\leq &\frac{1}{\gamma _{2}(c)-\gamma _{1}(c)}\left[ \int_{-\infty }^{\xi
}e^{(\gamma _{1}(c)+\mu )(\xi -s)}+\int_{\xi }^{\infty }e^{(\gamma
_{2}(c)-\mu )(\xi -s)}\right]  \\
&&\left( b^{\prime }(0)\beta KcT+\beta -d\right) \left\vert \phi
_{2}(s)-\phi _{1}(s)\right\vert e^{-\mu \left\vert s\right\vert }ds \\
&&+\frac{e^{\mu cM}b^{\prime }(0)}{\gamma _{2}(c)-\gamma _{1}(c)}\left[
\int_{-\infty }^{\xi }e^{(\gamma _{1}(c)+\mu )(\xi -s)}ds+\int_{\xi
}^{\infty }e^{(\gamma _{2}(c)-\mu )(\xi -s)}\right]  \\
&&\left\vert \phi _{2}(s-c\tau (\phi _{1}(s)))-\phi _{1}(s-c\tau (\phi
_{1}(s)))\right\vert e^{-\mu \left\vert s-c\tau (\phi _{1}(s))\right\vert }ds
\\
&\leq &\frac{b^{\prime }(0)\beta KcT+\beta -d+e^{\mu cM}b^{\prime }(0)}{%
\gamma _{2}(c)-\gamma _{1}(c)}\left[ \frac{1}{\gamma _{2}(c)-\mu }-\frac{1}{%
\gamma _{1}(c)+\mu }\right] \left\vert \phi _{2}-\phi _{1}\right\vert _{\mu
},
\end{eqnarray*}
which implies the continuity in the sense of $|\cdot |_{\mu}$.

Due to Lemma \ref{le2.1}, it is easy to prove the compactness and we omit the details. The proof is complete.
\end{proof}

We now present the main conclusion of this part.
\begin{theorem}\label{th2}
For every $c>c^*,$ \eqref{2} with \eqref{3} has a monotone solution.
\end{theorem}
\begin{proof}
Applying Schauder's fixed point theorem, the existence of fixed point of $F$ in $\Gamma$ is clear. Denote the fixed point by $\phi(\xi),$ then the monotonicity implies the existence of $\lim_{\xi\rightarrow -\infty }\phi (\xi), \lim_{\xi\rightarrow +\infty }\phi (\xi).$ Moreover, $\lim_{\xi\rightarrow -\infty }\phi (\xi)=0$ is obtained by the asymptotic behavior of upper and lower solutions. At the same time, the definition of $\underline{\phi}(\xi)$ leads to
\[
0< \lim_{\xi\rightarrow +\infty }\phi (\xi) \le K.
\]
Then it is easy to verify that $\lim_{\xi\rightarrow +\infty }\phi (\xi) = K.$ The proof is complete.
\end{proof}

\subsection{Existence of Monotone Traveling Wavefronts: $c=c^*$}

In this part, we shall establish the existence of monotone solutions of \eqref{2}-\eqref{3} if $c=c^*.$ We first present the main results as follows.
\begin{theorem}\label{th3}
If $c=c^*,$ then \eqref{2} with \eqref{3} has a monotone solution.
\end{theorem}
\begin{proof}
Let $\{c_n\}_{n\in \mathbb{N}}$ be a decreasing sequence satisfying
\[
c_n \to c^*, n\to \infty.
\]
Then Theorem \ref{th2} implies that for each $c_n,$ $F$ with $c=c_n$ has a fixed point $\phi_n(\xi),$ where $\phi_n(\xi)$ is monotone increasing.  Because a traveling wavefront is invariant in the sense of phase shift, we assume that
\begin{equation}\label{zh}
\phi_n(0)=K/2, n\in \mathbb{N}.
\end{equation}

From Lemma \ref{le2.2}, $\phi_n(\xi)$ is equicontinuous and uniformly bounded for all $n\in\mathbb{N},\xi\in\mathbb{R}.$ By Ascoli-Arzela Lemma and a standard nested subsequence argument, it follows that
there exists a subsequence of $\{c_n\}_{n\in \mathbb{N}},$ still denoted it by $\{c_n\}_{n\in \mathbb{N}},$ such that the corresponding $\phi_n(\xi)$  converges uniformly on every bounded interval of $\xi\in\mathbb{R},$ and hence pointwise for $\xi\in\mathbb{R}$ to a function $\phi^*(\xi).$ Moreover, it is evident that
\[
\min\{e^{\gamma _1 (c_n) (\xi-s)}, e^{\gamma _1 (c_n) (\xi-s)}\} \to \min\{e^{\gamma _1 (c^*) (\xi-s)}, e^{\gamma _1 (c^*) (\xi-s)}\}, n\to \infty
\]
and the convergence is uniform in $\xi, s\in \mathbb{R}.$

Let $n\to\infty$ in $F,$ then the dominated convergence theorem implies that
\[
F(\phi ^{\ast })(\xi )=\frac{1}{\gamma _{2}(c^{\ast })-\gamma _{1}(c^{\ast })%
}\left[ \int_{-\infty }^{\xi }e^{\gamma _{1}(c^{\ast })(\xi -s)}+\int_{\xi }^{\infty }e^{\gamma _{2}(c^{\ast })(\xi -s)}\right]H(\phi
^{\ast })(s)ds
\]
by the uniform continuity of $b(u), \tau (u).$ Therefore, $\phi ^{\ast }(\xi)$ is a fixed point of $F$ with $c=c^*,$ and a bounded solution of \eqref{2} with $c=c^*.$ In particular, Lemma \ref{le2.2} tells us that $\phi ^{\ast }(\xi)$ is uniformly continuous.

By \eqref{zh}, we see that $\phi ^{\ast }(0)=K/2.$ Moreover, the monotonicity and boundedness of $\phi_n(\xi)$ indicate that
$\phi ^{\ast }(\xi)$ is monotone increasing and bounded for $\xi\in \mathbb{R},$ which further implies that $\lim_{\xi\to\pm\infty}\phi ^{\ast }(\xi)$ exist. Denote
$$\lim_{\xi\to\pm\infty}\phi ^{\ast }(\xi)=\phi _{\pm}.$$
Then the uniform continuity and the dominated convergence theorem in $F$ lead to
\[
\phi _{\pm}=\frac{\beta \phi _{\pm} -d\phi _{\pm} +b(\phi _{\pm})}{\beta},
\]
and so $d\phi _{\pm} =b(\phi _{\pm}).$ Again by \eqref{zh}, we have
\[
0\le \phi_- \le K/2 \le \phi_+ \le K,
\]
which indicates that
\[
\phi_-=0, \phi_+=K
\]
by (B). The proof is complete.
\end{proof}

\subsection{Nonexistence of Monotone Traveling Wavefronts: $c<c^*$}

In this section, we investigate the nonexistence of traveling wavefronts if $c<c^*.$ We first consider the following initial value problem
\begin{equation}\label{4.1}
\begin{cases}
u_{t}(x,t)= u_{xx}(x,t)-D_1 u(x,t)+D_2 u(x,t-m)-D_3 u^2 (x,t),\\
u(x,s)=\psi (x,s),
\end{cases}
\end{equation}
in which  $x\in \mathbb{R}, t>0,s\in [-m, 0],$ $D_1, D_2, D_3$ are positive constants satisfying $D_2 >D_1$, $\psi (x,s)$ is uniformly continuous in $x\in \mathbb{R}, s\in [-m, 0].$ By Smith and Zhao \cite{sz}, Thieme and Zhao \cite{thieme}, we have the following three lemmas.

\begin{lemma}\label{le4.1}
If $\psi (x,s)$ satisfies
\[
0\le \psi (x,s) \le \frac{D_2- D_1}{ D_3}, x\in \mathbb{R}, s\in [-m, 0],
\]
then \eqref{4.1} has a unique mild solution $u(x,t)$ defined for all $x\in\mathbb{R},t>0.$ Moreover, $u(x,t)$ is continuous in
$x\in\mathbb{R},t>0,$ and satisfies
\[
0\le u(x,t)\le \frac{D_2- D_1}{ D_3}, x\in\mathbb{R},t>0.
\]
\end{lemma}
\begin{lemma}\label{le4.2}
Assume that
\[
0\le w(x,t) \le \frac{D_2- D_1}{ D_3}, x\in\mathbb{R},t\ge -m,
\]
and satisfies
\begin{equation}\label{4.2}
\begin{cases}
w_{t}(x,t)\ge (\le) w_{xx}(x,t)-D_1 w(x,t)+D_2 w(x,t-m)-D_3 w^2 (x,t),\\
w(x,s)\ge (\le) \psi (x,s)
\end{cases}
\end{equation}
for  $x\in \mathbb{R}, t>0,s\in [-m, 0].$ Then
$
w(x,t)\ge (\le) u(x,t), x\in\mathbb{R},t>0.
$
\end{lemma}
\begin{lemma}\label{le4.3}
Assume that $c>0$ such that
\[
\lambda^2 -c \lambda -D_1 +D_2e^{-\lambda c m}>0, \lambda \ge 0.
\]
If $\psi (x,s)$ satisfies
\[
0\le \psi (x,s) \le \frac{D_2- D_1}{ D_3}, x\in \mathbb{R}, s\in [-m, 0],\psi (0,0)>0,
\]
then
\[
\liminf_{t\to\infty}\inf_{|x|<ct} u(x,t)=\limsup_{t\to\infty}\sup_{|x|<ct} u(x,t)=\frac{D_2- D_1}{ D_3}.
\]
\end{lemma}

\begin{theorem}\label{th2no}
If $c<c^*,$ then \eqref{2} with \eqref{3} does not have a monotone solution.
\end{theorem}
\begin{proof}
Were the statement false, then  \eqref{2} with \eqref{3} has a monotone solution $\phi(\xi)$ for some $c_1 <c^*$. Evidently, $\phi(\xi)$ satisfies
\begin{equation}\label{4.3}
\lim_{\xi\to -\infty}\phi(\xi)=\lim_{\xi\to -\infty}\phi'(\xi)=0.
\end{equation}
In particular, we can select $\epsilon >0$ such that
\[
b'(0) > d+2\epsilon
\]
and
\[
\lambda^2 -c \lambda -(d+ \epsilon ) +b'(0)e^{-\lambda c m}>0
\]
for all $\lambda >0, 3c \le 2c^* +c_1.$ We first prove that there exists $N >0$ such that
\begin{equation}
b(\phi (\xi -c_{1}\tau (\phi (\xi ))))\geq -\epsilon \phi (\xi )+b^{\prime}(0)\phi(\xi - c_1 m) -N \phi ^2 (\xi ),\xi \in \mathbb{R}.  \label{4.4}
\end{equation}%
In fact, (B1)-(B4) and the monotonicity of $\phi (\xi )$ imply that
\begin{eqnarray*}
b(\phi (\xi -c_{1}\tau (\phi (\xi )))) &\geq &b^{\prime }(0)\phi (\xi
-c_{1}\tau (\phi (\xi )))-L\phi ^{2}(\xi -c_{1}\tau (\phi (\xi ))) \\
&\geq &b^{\prime }(0)\phi (\xi -c_{1}\tau (\phi (\xi )))-L\phi ^{2}(\xi ) \\
&=&b^{\prime }(0)\phi (\xi -c_{1}m)-L\phi ^{2}(\xi ) \\
&&+b^{\prime }(0)\left[ \phi (\xi -c_{1}\tau (\phi (\xi )))-\phi (\xi
-c_{1}\tau (0))\right] .
\end{eqnarray*}%
Due to \eqref{4.3}, there exists $\xi _{1}\in \mathbb{R}$ such that
\[
b^{\prime }(0)\left[ \phi (\xi -c_{1}\tau (\phi (\xi )))-\phi (\xi
-c_{1}\tau (0))\right] \geq -\epsilon \phi (\xi )
\]%
for all $\xi \leq \xi _{1}+c_{1}M.$ Therefore, we have
\[
b(\phi (\xi -c_{1}\tau (\phi (\xi )))\geq -\epsilon \phi (\xi )+b^{\prime
}(0)\phi (\xi -c_{1}m)-L\phi ^{2}(\xi )
\]%
for all $\xi \leq \xi _{1}+c_{1}M.$ If $\xi \geq \xi _{1}+c_{1}M,$ then
\begin{eqnarray*}
&&b^{\prime }(0)\left[ \phi (\xi -c_{1}\tau (\phi (\xi )))-\phi (\xi
-c_{1}\tau (0))\right]  \\
&\geq &-b^{\prime }(0)\phi (\xi -c_{1}\tau (0)) \\
&\geq &-b^{\prime }(0)\phi (\xi _{1}) \\
&=&\frac{-b^{\prime }(0)}{\phi (\xi _{1})}\phi ^{2}(\xi _{1}) \\
&\geq &\frac{-b^{\prime }(0)}{\phi (\xi _{1})}\phi ^{2}(\xi ),
\end{eqnarray*}%
which indicates \eqref{4.4}.

From \eqref{4.4}, $u(x,t)=\phi (x+c_1 t)$ satisfies
\begin{equation}\label{4.5}
\begin{cases}
u_{t}(x,t)\ge  u_{xx}(x,t)-(d+\epsilon ) u(x,t)+b'(0)u(x,t-m)- N u^2(x,t),\\
u(x,s)=\phi (x+ c_1s),
\end{cases}
\end{equation}
for $x\in \mathbb{R}, t>0,s\in [-m, 0].$
By Lemmas \ref{le4.1}-\ref{le4.3}, we see that
\[
\liminf_{t\to\infty} u(-(c_1+c^*)t/2,t) \ge \frac{b'(0)-(d+\epsilon )}{N} >0,
\]
which implies a contradiction because of
\[
\xi =-(c_1+c^*)t/2+c_1t \to -\infty, t\to \infty.
\]
The proof is complete.
\end{proof}

\section{Nonmonotone Birth Function}

In this section, we investigate the existence and nonexistence of positive traveling wave solutions of \eqref{1}, namely, existence and nonexistence of positive solutions of \eqref{2} and \eqref{4}, if
$b: [0,\infty)\to [0,\infty)$ satisfies the following assumptions:
\begin{description}
  \item[(C1)] $b(u)$ is $C^1$ for $u\in [0,\infty );$
  \item[(C2)] $b'(0)>d$ and $b(u)<b'(0)u, u\in [0,\infty );$
  \item[(C3)] $0\le b'(0)u-b(u)< Lu^2, |b'(u)|\le b'(0)$ for $u\in [0,K]$ and some $L>0.$
\end{description}
Clearly, if $b(u)=pue^{-u},$ then (C1)-(C3) hold when $p> d.$ In particular, (C1)-(C3) will be imposed throughout this section without further illustration. For convenience, we define
\[
\mathcal{K}=\max_{u\in [0, K]}b(u), \mathcal{T}=\sup_{u\in [0,\mathcal{K}]}\tau '(u).
\]

We first present the nonexistence of traveling wave solutions.
\begin{theorem}\label{th5.2}
If $c<c^*,$ then \eqref{2} does not have a positive solution $\phi(\xi)$ satisfying \eqref{4} and
$
0< \phi(\xi) \le \mathcal{K}.
$
\end{theorem}
\begin{proof}
Were the statement false, then for some $c_1 <c^*,$ \eqref{2} has a positive solution $\phi(\xi)$ satisfying \eqref{4}. Then $\phi(\xi)$ is a fixed point of $F$ and it is evident that
\[
\lim_{\xi\to-\infty}\phi'(\xi)=0.
\]

Similar to the proof of Theorem \ref{th2no}, we can select $\epsilon >0$ such that
\[
b'(0) > d+2\epsilon
\]
and
\[
\lambda^2 -c \lambda -(d+ \epsilon ) +b'(0)e^{-\lambda c m}>0
\]
for all $\lambda >0, 3c \le 2c^* +c_1.$

Then there exists $N>0$ such that
\[
b(\phi (\xi -c_{1}\tau (\phi (\xi ))))\geq -\epsilon \phi (\xi )+b^{\prime}(0)\phi(\xi - c_1 m) -N \phi ^2 (\xi ),\xi \in
\mathbb{R},
\]%
although $b(u)$ is not monotone. In fact, because $\phi ^{\prime }(\xi
)\rightarrow 0,\xi \rightarrow -\infty ,$ then there exists $\xi _{2}$ such
that
\[
b(\phi (\xi -c_{1}\tau (\phi (\xi ))))\geq b^{\prime }(0)\phi (\xi -c_{1}m)-%
\frac{\epsilon }{2}\phi (\xi ),\xi <\xi _{2}+c_{1}M.
\]%
If $\xi \geq \xi _{2}+c_{1}M,$ then $\inf_{\xi \geq \xi _{2}+c_{1}M}\phi
(\xi )>0$ such that there exists $N$ satisfying%
\[
N\phi ^{2}(\xi )\geq N\left( \inf_{t\geq \xi _{0}+c_{1}M}\phi (t)\right)
^{2}\geq \mathcal{K}\geq b(\phi (\xi -c_{1}\tau (\phi (\xi )))),\xi \geq \xi
_{0}+c_{1}M.
\]
Similar to the proof of Theorem \ref{th2no}, we can verify the result.
\end{proof}
\begin{remark}{\rm
Theorem \ref{th5.2} still holds if $b(u),u\in [0,K]$ is monotone.}
\end{remark}

For $u\in [0,\mathcal{K}],$ we define
\[
\overline{b}(u)=\max_{v\in [0,u]}b(v), \underline{b}(u)=\min_{v\in [u, \mathcal{K}]}b(v).
\]
Clearly, both $\overline{b}(u)$ and $\underline{b}(u)$ are monotone and continuous and there exists $k>0$ such that
\[
0< k\le K \le  {\mathcal{K}}
\]
and
\[
\underline{b}(k)=dk, \overline{b}(\mathcal{K})=d \mathcal{K}.
\]

Consider the following two equations
\begin{equation}\label{5.1}
V_{t}(x,t)= V_{xx}(x,t)-d V(x,t)+\overline{b}(V(x,t-\tau(V(x,t)))),
\end{equation}
and
\begin{equation}\label{5.2}
v_{t}(x,t)= v_{xx}(x,t)-d v(x,t)+\underline{b}(v(x,t-\tau(v(x,t)))),
\end{equation}
in which all the parameters are the same as those in \eqref{1}. Then  the existence and nonexistence of monotone traveling wavefronts of \eqref{5.1} and \eqref{5.2} can be answered by the conclusions in Section 2.
\begin{theorem}\label{th5.1}
If $c> c^*,$ then \eqref{2} admits a positive solution satisfying \eqref{4}.
\end{theorem}
\begin{proof}
For any fixed $c>c^*,$   we define
\begin{equation}\label{5.3}
\overline{\phi }(\xi )=\min \{e^{\lambda _{1}(c)\xi },\mathcal{K}\},\underline{\phi }%
(\xi )=\varphi (\xi),
\end{equation}
where $\lambda _{1}(c)$ is the same as that in Section 2, $\varphi (\xi)$ is a monotone traveling wavefront of \eqref{5.2} and satisfies
\[
\lim_{\xi\to - \infty}\varphi (\xi)e^{- \lambda _{1}(c)\xi} =1.
\]
By the discussion in Section 2, we see that
\[
\overline{\phi }(\xi ) \ge \underline{\phi }(\xi ), \xi\in\mathbb{R}.
\]
Define
\[
\Gamma ^*\left( \underline{\phi },\overline{\phi }\right) =\left\{
\phi (\xi)\in B_\mu \left( \mathbb{R},\mathbb{R}\right):
\begin{array}{l}
(i)\text{ }\underline{\phi }(\xi)\leq \phi (\xi)\leq \overline{\phi }(\xi); \\
(ii)\text{ } c |\phi (\xi_1)-\phi(\xi_2)| \le \frac{\beta \mathcal{K}}{1+b'(0)\mathcal{K}\mathcal{T}} |\xi_1-\xi_2|, \xi_1,\xi_2\in \mathbb{R}
\end{array}
\right\}.
\]
Similar  to the discussion in Section 2, we can select $\beta, \mu$ such that
\begin{description}
  \item[(P1)] $\Gamma ^*$ is nonempty, convex, bounded and closed;
  \item[(P2)] $F: \Gamma ^* \to \Gamma ^*;$
  \item[(P3)] $F: \Gamma ^* \to \Gamma ^*$ is complete continuous in the sense of decay norm $|\cdot|_{\mu},$
\end{description}
in which the definition of $F$ is the same as those in Section 2. Then there exists $\phi\in \Gamma ^*$ such that $\phi$ is a fixed point of $F$ and satisfies
\[
k\le \liminf_{\xi\to\infty} \phi(\xi)\le \limsup_{\xi\to\infty} \phi(\xi)\le \mathcal{K},
\]
which implies \eqref{4}. This completes the proof.
\end{proof}

\begin{theorem}\label{th5.3}
If $c= c^*,$ then \eqref{2} admits a nontrivial  positive solution satisfying
\begin{equation}
k\le \liminf_{\xi\rightarrow +\infty
}\phi (\xi)\le \limsup_{\xi\rightarrow +\infty
}\phi (\xi) \le \mathcal{K}.
\end{equation}
\end{theorem}
\begin{proof}
Let $\delta \in (0, k/2)$ be a positive constant and $\{c_n\}_{n\in \mathbb{N}}$ be a decreasing sequence satisfying
\[
c_n \to c^*, n\to \infty.
\]
Then Theorem \ref{th5.1} implies that for each $c_n,$ $F$ with $c=c_n$ has a fixed point $\phi_n(\xi)$ such that
\begin{equation}\label{5.4}
\phi_n(0)=  \delta,\phi_n(\xi) \ge  \delta, \xi \ge 0, n\in \mathbb{N},
\end{equation}
then we can obtain the existence of $\phi^*(\xi)$ such that $\phi^*(\xi)$ is a positive solution of \eqref{2} with $c=c^*$ by a discussion similar to that in Section 2.

Denote
\[
\liminf_{\xi\to +\infty}\phi^*(\xi)= \phi^- (\ge  \delta), \limsup_{\xi\to +\infty}\phi^*(\xi)= \phi^+ (\ge  \delta).
\]
By the dominated convergence theorem in $F,$ we see that
\[
\phi^- \ge \frac{\beta \phi^- -d \phi^- +\underline{b}(\phi^-)}{\beta}
\]
and
\[
\phi^+ \le \frac{\beta \phi^+ -d \phi^+ +\overline{b}(\phi^+)}{\beta}.
\]
Then \eqref{5.4} is true. Because of $\phi^*(0) < \frac{3 \phi^-}{4},$ it is evident that $\phi^*(\xi)$ is not a constant. The proof is complete.
\end{proof}

\begin{remark}{\rm
Similar to those in \cite{fz,ma1,why}, we can obtain some sufficient conditions such that $\lim_{\xi\to\infty}\phi(\xi)$ exists even if $b$ is not monotone.}
\end{remark}

Unfortunately, we cannot obtain the existence of $\lim_{\xi\to -\infty}\phi(\xi)$ in Theorem \ref{th5.3}. However, under some additional assumptions, we can formulate the limit behavior as follows.
\begin{theorem}\label{th000}
Assume that $\mathcal{T}\ge 0$ is small enough. If $\phi (\xi)$ is defined by Theorem \ref{th5.3}, then $\lim_{\xi\to -\infty}\phi(\xi)=0.$
\end{theorem}
\begin{proof}
Similar to Lemma \ref{le2.2}, if we take
\[
4 \beta = (1+b^{\prime }(0)\mathcal{K}\mathcal{T})^2 {c^*}^2 + 4d,
\]
then
\[
|c^*\phi '(\xi)|<\frac{\beta K}{1+ b^{\prime }(0)\mathcal{K}}=: L_3.
\]
Therefore, we have
\[
\left\vert b(\phi (\xi -c^* \tau (\phi (\xi ))))-b(\phi (\xi -cm))\right\vert
\leq b^{\prime }(0)L_{3}\mathcal{T}\phi (\xi )
\]
and so
\begin{eqnarray*}
&&-d\phi (\xi ) +b(\phi (\xi -c^* \tau (\phi (\xi ))))\\
&\geq &-(d+b^{\prime }(0)L_{3}\mathcal{T})\phi (\xi )+b(\phi (\xi -cm)).
\end{eqnarray*}

If $\mathcal{T}>0$ such that
\[
d+b^{\prime }(0)L_{3}\mathcal{T} <b'(0),
\]
then there exists $\underline{k} \in (0,k]$ such that
\[
(d+b^{\prime }(0)L_{3}\mathcal{T})\underline{k}=\underline{b}(\underline{k}).
\]

Take $4\delta =\underline{k}$ in the proof of Theorem \ref{th5.3}. If $\limsup_{\xi\to -\infty}\phi(\xi)>0,$  then there exist
a strictly decreasing sequence $\{\xi_j\}$ and a constant $\epsilon >0$ such that
\[
\xi_j<0, \xi_j \to -\infty, j\to \infty
\]
and
\[
2\phi (\xi ) >\limsup_{\xi\to -\infty}\phi(\xi), \xi \in [\xi_j -\epsilon, \xi_j +\epsilon], j\in \mathbb{N}.
\]

Note that $u(x,t)=\phi (\xi )$ is an upper solution of
\begin{equation*}
u_{t}(x,t)\ge  u_{xx}(x,t)-(d+b^{\prime }(0)L_{3}\mathcal{T})\phi (\xi )+b(\phi (\xi -cm)),
\end{equation*}
then the standard comparison principle and asymptotic spreading \cite{thieme} indicate that
\[
4\phi (\xi_j )>3 \underline{k},
\]
and a contradiction occurs. The proof is complete.
\end{proof}
\begin{remark}{\rm
The proof of Theorem \ref{th000} is similar to Lin and Ruan \cite[Section 5.5]{lr}.}
\end{remark}

\section*{Acknowledgments}
Guo Lin is grateful to Professor  Xingfu Zou for his very valuable comments and suggestions.


\begin{thebibliography}{10}\setlength{\itemsep}{0mm}\linespread{1}\selectfont
\bibitem{afw} W.G. Aiello, H.I. Freedman, J. Wu, Analysis of a model representing stage-structured population growth with state-dependent time delay, \emph{SIAM J. Appl. Math.,}   \textbf{52} (1992), 855-869.

\bibitem{and} H.G. Andrewartha, L.C. Birch, The Distribution and Abundance of Animals, University
of Chicago Press, Chicago, IL, 1954.

\bibitem{ari} O. Arino, K.P. Hadeler, M.L. Hbid, Existence of periodic solutions for delay differential equations with state dependent delay, \emph{J. Differential Equations,} \textbf{144} (1998), 263-301.

\bibitem{coo} K.L. Cooke, W. Huang, On the problem of linearization for state-dependent delay differential equations, \emph{Proc. Amer. Math. Soc.,} \textbf{124} (1996) 1417-1426.

\bibitem{fz}
J. Fang, X.-Q. Zhao, Existence and uniqueness of traveling waves for non-monotone integral equations with applications, \emph{J. Differential Equations,} \textbf{248} (2010), 2199-2226.

\bibitem{har}  F. Hartung, T. Krisztin, H.-O. Walther, and J. Wu, Functional differential equations with state-dependent delays: Theory and applications, in \emph{``Handbook of Differential Equations: Ordinary Differential Equations"} (ed, by A. Canada), Elsevier, Dordrecht, The Netherlands, 2006.

\bibitem{hu} Q. Hu, J.  Wu, X. Zou, Estimates of periods and global continua of periodic solutions for state-dependent delay equations, \emph{SIAM J. Math. Anal.,}  \textbf{44}  (2012),   2401-2427.

\bibitem{liangzhao}  X. Liang,  X.Q. Zhao, Asymptotic speeds of spread and traveling waves
for monotone semiflows with applications, \emph{Comm. Pure Appl.
Math.,} \textbf{60} (2007), 1-40.

\bibitem{lr} G. Lin, S. Ruan, Traveling wave solutions for delayed reaction-diffusion systems and applications to Lotka-Volterra competition-diffusion models with distributed delays,  \emph{J. Dynam. Diff. Eqns.,} in press.

\bibitem{ma01} S. Ma, Traveling wavefronts for delayed
reaction-diffusion systems via a fixed point theorem, \emph{J.
Differential Equations,} \textbf{171} (2001), 294-314.

\bibitem{ma1} S. Ma, Traveling waves for non-local delayed diffusion equations via
auxiliary equations, \emph{J. Differential Equations,} \textbf{237}
(2007), 259-277.

\bibitem{magal}  P. Magal, O. Arino, Existence of periodic solutions for a state-dependent delay differential equation, \emph{J. Differential Equations,} \textbf{165} (2000),   61-95

\bibitem{mal} J. Mallet-Paret, R.D. Nussbaum, Superstability and rigorous asymptotics in singularly perturbed state-dependent delay-differential equations, \emph{J. Differential Equations,}  \textbf{250}  (2011),  4037-4084.

\bibitem{schaaf} K.W. Schaaf, Asymptotic behavior and traveling wave
solutions for parabolic functional differential equations,
\emph{Trans. Amer. Math. Soc.,} \textbf{302} (1987), 587-615.

\bibitem{sz} H.L. Smith, X.Q. Zhao,  Global asymptotic stability of
traveling waves in delayed reaction-diffusion equations, \emph{SIAM
J. Math. Anal.,} \textbf{31} (2000), 514-534.

\bibitem{thieme} H.R. Thieme, X.Q. Zhao, Asymptotic speeds of spread and traveling
waves for integral equations and delayed reaction-diffusion models,
\emph{J. Differential Equations,} \textbf{195} (2003), 430-470.

\bibitem{wa}  H.O. Walther, Semiflows for neutral equations with state-dependent delays, in \emph{``Infinite dimensional dynamical systems"} (ed. by J. Mallet-Paret, J. Wu, Y. Yi and H. Zhu),  211-267, Fields Inst. Commun., 64, Springer, New York, 2013.

\bibitem{why}  H. Wang, On the existence of traveling waves
for delayed reaction-diffusion equations, \emph{J. Differential
Equations,} \textbf{247} (2009),  887-905.

\bibitem{wuzou2} J. Wu, X. Zou, Traveling wave fronts of
reaction-diffusion systems with delay, \emph{J. Dynam. Diff. Eqns.,}
\textbf{13} (2001), 651-687.
\end{thebibliography}
\end{document}